\documentclass[pdflatex,sn-mathphys-num]{sn-jnl}

\usepackage{graphicx}
\usepackage{amsmath,amssymb,amsfonts,mathrsfs,cancel,amsthm}
\usepackage{multirow,xcolor,textcomp,booktabs}
\usepackage{manyfoot,listings}
\usepackage[title]{appendix}
\usepackage{tcolorbox,bm,bbm}
\usepackage{algorithm2e}
\usepackage{tikz,pgfplots}
\usepackage{natbib}
\pgfplotsset{compat=newest}
\usetikzlibrary{arrows.meta}

\newcommand{\ren}{\mathbb{R}^n}
\newcommand{\re}{\mathbb{R}}

\newcommand{\blockcomment}[1]{}

\theoremstyle{thmstyleone}%
\newtheorem{theorem}{Theorem}
\newtheorem{lemma}[theorem]{Lemma}

\theoremstyle{thmstyletwo}%
\newtheorem{example}{Example}%
\newtheorem{remark}{Remark}%

\theoremstyle{thmstylethree}%
\newtheorem{definition}{Definition}%

\begin{document}

\title[Article Title]{A Certificate of Unboundedness for Polynomial Optimization Problems}


\author*[1]{\fnm{Rohan} \sur{Rele}}\email{rrele@asu.edu}

\author[1]{\fnm{Angelia} \sur{Nedi\'c}}\email{angelia.nedich@asu.edu}
\equalcont{These authors contributed equally to this work.}

\affil*[1]{\orgdiv{School of Electrical, Computer and Energy Engineering}, \orgname{Arizona State University}}


\abstract{Global polynomial optimization methods typically rely on compactness of the feasible region in order to find solutions. These methods can incur considerable computational expense and most commercially available solvers do not verify the existence of a solution prior to undergoing global search. In this manuscript we propose a simple pre-processing algorithm to determine if an arbitrary polynomial optimization problem is unbounded from below thereby providing information about the problem's asymptotic geometry prior to solving the problem if a solution can be found. }




\maketitle

\section{Introduction}\label{sec1}

Polynomial optimization \cite{Lasserre2001Global, Lasserre2015Introduction} 
arises in many engineering disciplines including control 
\cite{PauwelsHenrionLasserre2014}, 
statistics \cite{WuYangPolynomialStatistics, AyvazDeLathauwer2022}, 
operations research \cite{AhmadiMajumdar2016}, 
and robotics 
\cite{AhmadiHallMakadiaSindhwani2017, HolmesDumbgenBarfoot2024}. 
While the problem has been well studied from the lens of semialgebraic geometry 
\cite{BlekhermanParriloThomas2013, DrusvyatskiyLewis2013}, 
global methods for solving generic polynomial programs in moderate 
dimensions remain computationally challenging.  
Current methods fall broadly into two families:  
(i) Sum-of-Squares and moment relaxations 
\cite{Lasserre2001Global, Lasserre2015Introduction, Laurent2009SOS}, and  
(ii) Branch-and-Bound, or branch-and-reduce, algorithms 
\cite{Morrison2016BranchBoundSurvey}.

The convergence guarantees of both SOS hierarchies and branch and bound methods rely on the assumption that 
either the feasible set or the sublevel sets of the objective are compact.  
As highlighted in~\cite{Lasserre2015Introduction, Laurent2009SOS, Schweighofer2005}, 
this assumption frequently fails in practical polynomial systems. Some modern solvers include methods to determine whether the problem is unbounded from below during optimization; however, no modern solver explicitly computes rays upon which the objective diverges to $-\infty$ prior to global search.  Detecting such directions provides a mathematical certificate by which to definitively assert unboundedness of the problem. Without early detection solvers cannot certify a finite lower bound and global routines may hang, numerically fail or return artificial optima. Recently, \cite{vu2025lowerboundedness} introduced an algebraic approach for decidably determining whether a polynomial program is lower bounded; however, this method is computationally infeasible beyond very small instances. Motivated by this gap, we introduce a fast certificate that detects unbounded directions in polynomial optimization problems.

\subsection*{Notation and basic definitions}

For a set $X \subseteq \ren$ its \textit{indicator function} is given by
\[ 
\delta_X(x): = \begin{cases}
\begin{aligned}
    0 \quad x &\in X, \\
    +\infty \quad x &\notin X.
    \end{aligned}
\end{cases}\]
The \emph{asymptotic cone} \cite{auslender_teboulle_2002} of a closed set $X$ is 
\[
    X_\infty := \left\{ d\in\mathbb{R}^n : 
    \exists \{t_k\} \subseteq \re, \{x_k\} \subseteq X, t_k \to \infty \text{ and } 
    \frac{x_k}{t_k} \to d\text{ as } k \to \infty \right \}.
\]
A function $f: \ren \to \re \cup \{+\infty\}$ is called \textit{proper} on $X$ if the set 
\[
\text{dom}~f := \{ x \in X : f(x) < +\infty\}
\] is nonempty. 
For a proper function $f$, the \emph{asymptotic function} has the following closed form 
\[
    f_\infty(d) := 
    \liminf_{\substack{t\to\infty\\ d'\to d}}
        \frac{f(t d')}{t}\] 
        according to Definition 2.5.1  and Theorem 2.5.1 in~\cite{auslender_teboulle_2002}. 

\subsection*{Asymptotic certificate of unboundedness}

We consider the polynomial optimization problem
\begin{equation}\tag{POP}\label{POP}
    \inf_{x\in X} f(x),
\end{equation}
with a polynomial objective function $f: \ren \to \re$ and the set $X$ given by polynomial functions $g_i: \ren \to \re$ polynomial functions for $i \in \{1,\ldots,m\}$,
\[
    X := \{x\in\mathbb{R}^n : g_i(x)\le 0,\ i\in[m]\},
\]
where $[m]:=\{1,\ldots,m\}$. We assume that the set $X$ is nonempty.
Asymptotic analysis
\cite{auslender_teboulle_2002, rele2025existencesolutionsnonconvexminimization} provides the foundational framework for analyzing unboundedness in optimization problems. The word `unboundedness' has two-fold meaning: if the set $X$ is bounded, then trivially \eqref{POP} must be bounded from below and thus when we refer to `unboundedness' we refer to the case of the set $X$ being unbounded \textit{and} the problem \eqref{POP} being unbounded from below i.e., $\inf_{x \in X} f(x) = -\infty$. In the subsequent development we apply a well-known necessary condition for boundedness from below \cite{auslender_teboulle_2002, rele2025existencesolutionsnonconvexminimization, ButtazzoTomarelli1991} to the case of $F(x):= (f + \delta_X)(x)$. More specifically
\[ \text{if there exists } d \in X_\infty \backslash \{0\} : F_\infty(d) < 0 \implies \inf_{x \in X} f(x) = -\infty. \] By a property of the asymptotic function \cite{auslender_teboulle_2002} we have that $F_\infty(d) \geq f_\infty(d) + \delta_{X_\infty}(d) $ which simplifies the above statement to 
\begin{equation}
    \text{if there exists } d \in X_\infty \backslash \{0\} : f_\infty(d) < 0 \implies \inf_{x \in X} f(x) = -\infty. \label{logical-implication}
\end{equation}

Note that the statement \eqref{logical-implication} is not specific to polynomial problems but we apply it to \eqref{POP} in what follows. 

\subsection*{Contribution}

We introduce a fast, sampling-based certificate that tests condition \eqref{logical-implication} under a weaker hypothesis that holds specifically for polynomial optimization problems.  
The approach to checking individual directions is inexpensive, can be used as a solver-agnostic preprocessing tool, and detects unboundedness by explicitly identifying feasible descent directions at infinity.  
When such directions occupy a set of positive normalized surface measure on $\mathbb{S}^{n-1}$, the certificate detects unboundedness with high probability under random sampling.  
When no certifying direction is found, we provide statistical guarantees quantifying the likelihood that untested directions could yield unboundedness.

Section~2 develops the asymptotic framework underlying condition \eqref{logical-implication}.  
Section~3 presents the certification procedure, its computational complexity, sampling paradigms and confidence bounds on the number of samples required to accept an inconclusive outcome.  
Section~4 characterizes degenerate regimes in which certifying directions exist but occupy a set of zero measure.  
Section~5 discusses the implications of these results for unboundedness detection in global polynomial optimization solvers.

\section{Polynomial Asymptotics and Unboundedness}

We begin by restating Lemma~3.1 of \cite{rele2025existencesolutionsnonconvexminimization} 
in a form that emphasizes an equation used later in our analysis.

\begin{lemma}[Polynomial asymptotics along a ray]
\label{lem:poly-asymptotics}
Let $h:\mathbb{R}^n\to\mathbb{R}$ be a polynomial of degree $p$ with 
a unique decomposition 
\[
h(x)=\sum_{k=0}^p \phi_k(x),
\]
where each $\phi_k:\mathbb{R}^n\to\mathbb{R}$ is a homogeneous polynomial of degree $k$. 
Fix a direction $d\in\mathbb{R}^n\setminus\{0\}$ and define 
\[
\mu(d)=\max\{k : \phi_k(d)\neq 0\}.
\]
Then
\begin{equation}\label{poly-form}
h(t d)
=
t^{\mu(d)}\phi_{\mu(d)}(d)
\;+\;
\sum_{k=0}^{\mu(d)-1} t^k \phi_k(d)
\end{equation}
and the lower-order sum satisfies
\[
\sum_{k=0}^{\mu(d)-1} t^k \phi_k(d)
=
o\!\big(t^{\mu(d)}\big).
\]
Consequently,
\[
h_\infty(d)=
\begin{cases}
-\infty & \mu(d)\ge 2\ \text{and }\phi_{\mu(d)}(d)<0,\\[0.25em]
\phi_1(d) & \mu(d)=1,\\[0.25em]
+\infty & \mu(d)\ge 2\ \text{and }\phi_{\mu(d)}(d)>0.
\end{cases}
\]
\end{lemma}

\begin{proof}
Since each $\phi_k$ is homogeneous of degree $k$, we have 
$\phi_k(t d)=t^k\phi_k(d)$ for all $t\in\mathbb{R}$.
Hence
\[
h(td)=\sum_{k=0}^{p} t^k\phi_k(d)
      =t^{\mu(d)}\phi_{\mu(d)}(d)
       +\sum_{k=0}^{\mu(d)-1} t^k\phi_k(d),
\]
which establishes \eqref{poly-form}.  
Factoring out $t^{\mu(d)}$ shows that the remaining terms vanish after scaling:
\[
\frac{1}{t^{\mu(d)}}\sum_{k=0}^{\mu(d)-1} t^k\phi_k(d)
=
\sum_{k=0}^{\mu(d)-1} t^{k-\mu(d)}\phi_k(d)
\longrightarrow 0,
\]
since $k-\mu(d)<0$ for all terms.
Dividing $h(td)$ by $t$ and taking limits yields the formula for $h_\infty(d)$, with the extended limit 
$\pm\infty$ determined by the sign of $\phi_{\mu(d)}(d)$ when $\mu(d)\ge 2$, and the finite case given by $\phi_1(d)$ when $\mu(d)=1$.
\end{proof}

\begin{remark}[Degree parity is not informative in multiple variables]
For univariate polynomials $(n=1)$, the parity of the leading degree determines the sign of the function
as $|x|\to\infty$.  
When $n\ge 2$, this is no longer true: an even-degree homogeneous polynomial may take both positive and negative values depending on the direction.  
For example,
\[
\phi_4(x_1,x_2)=x_1^4-x_2^4
\]
is of even degree but is indefinite.  
Thus the sign of $h_\infty(d)$ depends on the direction $d$ rather than on the degree parity of $h$.
\end{remark}

Next we introduce a lemma which motivates the eventual claim of unboundedness. 

\begin{lemma}[Eventual feasibility]\label{lem:eventual-feasibility}
Let $h:\mathbb{R}^n\to\mathbb{R}$ be a polynomial and let $d\in\mathbb{R}^n\setminus\{0\}$ be a direction.
If $h_\infty(d) < 0$, then
\[
h(td) \to -\infty \quad \text{as } t \to \infty,
\]
and in particular there exists $T>0$ such that
\[
h(td) < 0 \quad \forall\, t \ge T.
\]
\end{lemma}

\begin{proof}
Write $h(x) = \sum_{k=0}^p \phi_k(x)$ with each $\phi_k$ homogeneous of degree $k$ and let
\[
\mu(d) := \max\{k : \phi_k(d)\neq 0\}.
\]
By Lemma~\ref{lem:poly-asymptotics} the condition $h_\infty(d)<0$ can only occur in two cases:
\begin{equation}\label{eq:mu-cases}
\mu(d)\ge 2 \ \text{and}\ \phi_{\mu(d)}(d)<0
\qquad\text{or}\qquad
\mu(d)=1 \ \text{and}\ \phi_1(d)<0.
\end{equation}

We treat the two cases in \eqref{eq:mu-cases} separately.

\smallskip
\noindent\emph{Case 1: $\mu(d)\ge 2$ and $\phi_{\mu(d)}(d)<0$.}
By Lemma~\ref{lem:poly-asymptotics},
\[
h(td) = t^{\mu(d)}\phi_{\mu(d)}(d) + \sum_{k=0}^{\mu(d)-1} t^k \phi_k(d),
\]
so $t\mapsto h(td)$ is a univariate polynomial of degree $\mu(d)\ge 2$ with negative leading coefficient $\phi_{\mu(d)}(d)<0$. Hence $h(td)\to -\infty$ as $t\to\infty$, and there exists $T_1>0$ such that $h(td)<0$ for all $t\ge T_1$.

\smallskip
\noindent\emph{Case 2: $\mu(d)=1$ and $\phi_1(d)<0$.}
In this case $\phi_k(d)=0$ for all $k\ge 2$, so along the ray
\[
h(td) = t\phi_1(d) + \phi_0(d),
\]
an affine function of $t$ with negative slope $\phi_1(d)<0$. Thus $h(td)\to -\infty$ as $t\to\infty$, and there exists $T_2>0$ such that $h(td)<0$ for all $t\ge T_2$.

\smallskip
Setting $T := \max\{T_1,T_2\}$ gives
\[
h(td)\to -\infty \text{ as } t\to \infty
\quad\text{and}\quad
h(td)<0 \ \ \forall\, t\ge T.
\]
\end{proof}

From these two lemmas we prove the main result of this paper which underlies the certificate presented in the following section. 

\begin{theorem}[Directional certificate for unboundedness]\label{thm:directional-certificate}
Let $f, g_1,\ldots,g_m : \mathbb{R}^n \to \mathbb{R}$ be polynomials and define
\[
X := \{x \in \mathbb{R}^n : g_i(x) \le 0,\ i \in [m] \} \quad \text{as in}~\eqref{POP}.
\]
Suppose there exists a direction $d \in \mathbb{R}^n \setminus \{0\}$ such that
\begin{equation}\label{key-condition}
f_\infty(d) < 0,
\qquad
(g_i)_\infty(d) < 0 \quad \text{for all } i \in [m].
\end{equation}
Then the polynomial optimization problem
\[
\inf_{x\in X} f(x)
\]
is unbounded below; that is,
\[
\inf_{x\in X} f(x) = -\infty.
\]
\end{theorem}

\begin{proof}
Fix such a direction $d$. By Lemma~\ref{lem:eventual-feasibility}, the condition $(g_i)_\infty(d) < 0$
ensures the existence of $T_i > 0$ such that 
\[
g_i(td) < 0 \qquad \forall\, t \ge T_i.
\]
Hence the ray $\{td : t \ge T\}$, where $T := \max_i T_i$, lies in the feasible set $X$.

Applying Lemma~\ref{lem:eventual-feasibility} to $f$ as well yields a $T_f > 0$ such that
\[
f(td) < 0 \qquad \forall\, t \ge T_f,
\qquad\text{and}\qquad 
f(td) \to -\infty \ \text{as } t \to \infty.
\]
Let $T' := \max\{T, T_f\}$. For all $t \ge T'$, the point $td$ is feasible and satisfies $f(td) < 0$, and along this feasible ray we have $f(td) \to -\infty$. Therefore,
\[
\inf_{x\in X} f(x) \le f(td) \to -\infty,
\]
which proves that the problem is unbounded below.
\end{proof}

Note that for a generic (non-polynomial) optimization problem the condition 
\[ \exists d\neq 0 \quad \text{such that}~ f_\infty(d) <0 \text{ and } d \in X_\infty\] is sufficient to prove unboundedness from below. However determining the membership $d \in X_\infty$
is not trivial and hence the asymptotic feasibility condition in \eqref{key-condition} 
\[ (g_i)_\infty(d) < 0 \quad \text{for } i \in [m] \] allows us to bypass this. We conclude this section with a remark on the case of directions with zero asymptotic growth. 

\begin{remark}[Directions of Zero Asymptotic Growth]
Note that by Lemma \ref{lem:poly-asymptotics}, the case $f_\infty(d)=0$ can occur only when
all homogeneous components of degree $\ge 2$ vanish at $d$, and the linear term
satisfies $\phi_1(d)=0$. Thus the set $\{d : f_\infty(d)=0\}$ is exactly the 
zero set of these finitely many homogeneous components.
No additional geometric behavior arises: these 
directions are simply ``flat'' at infinity and are not informative in determining whether or not \eqref{POP} is unbounded from below. 
\end{remark}

In the next section we turn Theorem \ref{thm:directional-certificate} into a practically implementable algorithm. 

\section{Certificate Procedure}

In the previous section we prove Theorem \ref{thm:directional-certificate} which gives an analytical criterion for detecting unboundedness. In this section we present an implementation of this idea which is based on directional sampling. That is, for a polynomial objective $f$ and constraint polynomials $\{g_i\}$ as in \eqref{POP}, if we are able to find a direction $d \in \ren \backslash \{0\}$ which satisfies \eqref{key-condition} then we can certify that \eqref{POP} is \textit{not} bounded from below. Our Lemma \ref{lem:poly-asymptotics} gives us a way to compute $f_\infty$ and $(g_i)_\infty$ for all $i \in [m]$ but we need one more conceptual bridge to be able to fully realize this certificate. 

\begin{remark}[Normalization of certificate directions]\label{rem:normalization}
The certificate in Theorem~\ref{thm:directional-certificate} depends only on the 
direction of $d \in \ren \setminus \{0\}$ and not on its magnitude. 
Let $d \neq 0$ and define the normalized direction $\hat d := d / \|d\|$. 
Fix a polynomial $h$ and write $h(x) = \sum_{k=0}^p \phi_k(x)$ with each $\phi_k$ 
homogeneous of degree $k$, as in Lemma~\ref{lem:poly-asymptotics}. For this $h$ and
direction $d$, define
\[
\mu(d) := \max\{k : \phi_k(d)\neq 0\}.
\]
Since each $\phi_k$ is homogeneous of degree $k$, we have
\[
\phi_k(\hat d)
=
\phi_k\!\left(\frac{d}{\|d\|}\right)
=
\frac{1}{\|d\|^k}\,\phi_k(d),
\]
so in particular $\phi_k(\hat d)=0$ if and only if $\phi_k(d)=0$. Hence
$\mu(\hat d)=\mu(d)$ and
\[
\phi_{\mu(d)}(\hat d)
=
\frac{1}{\|d\|^{\mu(d)}}\,\phi_{\mu(d)}(d),
\]
so $\phi_{\mu(d)}(\hat d)$ and $\phi_{\mu(d)}(d)$ have the same sign. One can now apply Lemma~\ref{lem:poly-asymptotics} and perform a simple case-argument to see that the signs of $h_\infty(d)$ and $h_\infty(\hat d)$ coincide. This justifies restricting the
algorithmic search to the unit sphere.
\end{remark}

Now given polynomials $f, g_1, \ldots, g_m : \ren \to \re$ we propose the following certification procedure. 

\begin{algorithm}[H]\label{alg:main}
\caption{Directional Certificate of Unboundedness for Polynomial Optimization Problems}
\vspace{0.5em}
\DontPrintSemicolon

\KwIn{
    Objective polynomial $f$; constraint polynomials $g_1,\dots,g_m$; sampling budget $N$; 
}
\KwOut{
    Either ``\textsc{unbounded}'' or ``\textsc{inconclusive}''.
}

\BlankLine

\textbf{Step 0.} 
Sample a set of $N$ directions according to some distribution from the sphere $\mathbb{S}^{n-1}$ to form a set $\{d_1, d_2, \ldots, d_N\} := \mathcal{D}.$ \\
\textbf{Step 1. }
Decompose each polynomial $h \in \{f,g_1,\dots,g_m\}$ into homogeneous parts
$h = \sum_{k} \phi_{h,k}$ 
where $k_h := \max\{k : \phi_{h,k}(d) \neq 0\}$. \\ 
\textbf{Step 2.}
\ForEach{$d \in \mathcal{D}$}{
    \textbf{2a. Check asymptotic feasibility.}
    The direction $d$ is asymptotically feasible if
    \[
        (g_i)_\infty(d) < 0 \qquad \forall\, i = 1,\dots,m.
    \]
    \textbf{2b. Apply the certificate condition.}
    \If{$f_\infty(d) < 0$ and $d$ is asymptotically feasible}{
        \Return{``\textsc{unbounded}''}
    }
}

\BlankLine
\Return{``\textsc{inconclusive}''}
\end{algorithm}

\bigskip 

Algorithm~\ref{alg:main} is, by design, a one-sided certificate for
unboundedness. Whenever a sampled direction $d \in \mathcal{D}$ satisfies
\[
    f_\infty(d) < 0
    \qquad\text{and}\qquad 
    (g_i)_\infty(d) < 0 \quad \forall\, i \in [m],
\]
the algorithm returns \textsc{unbounded}. Theorem~\ref{thm:directional-certificate}
serves as a proof of correctness for this output: such a direction certifies the
existence of a feasible ray along which $f$ diverges to $-\infty$, and therefore
$\inf_{x \in X} f(x) = -\infty$.

At the same time, the certificate is inherently incomplete. If no violating
direction is found in the finite sample set $\mathcal{D}$, the algorithm returns
\textsc{inconclusive}. This outcome does \emph{not} assert that the problem is
bounded from below. The reason is straightforward: the algorithm inspects only a
finite subset $\mathcal{D}\subset \mathbb{S}^{n-1}$, and the absence of a certifying
direction in $\mathcal{D}$ does not rule out the possibility that some unsampled direction
$d^\star \in \mathbb{S}^{n-1}$ satisfies the conditions \eqref{key-condition}. In the subsections that follow we discuss sampling considerations, complexity and develop a probabilistic interpretation of an inconclusive outcome.

\subsection{Sampling and Complexity}
\label{subsec:complexity-sampling}

The performance of Algorithm~\ref{alg:main} depends on two
components: the computational cost of evaluating the quantities
$f_\infty(d)$ and $(g_i)_\infty(d)$, and the strategy used to construct the
finite sampling set $\mathcal{D} \subset \mathbb{S}^{n-1}$. We discuss each in turn.

\paragraph{Computational complexity.}
For each polynomial $h \in \{f, g_1,\dots,g_m\}$ let $M_h$ denote the number
of component functions ($\phi_i)$ appearing in $h$. Grouping the components of $h$
by total degree to form its homogeneous components requires a single pass
over the list of $\phi_i$'s and therefore costs $O(M_h)$ operations. This need only be done once ideally prior to sampling. Once this
decomposition is available, evaluating  $h_\infty(d)$ for
a given direction $d \in \mathbb{S}^{n-1}$ requires evaluating the homogeneous
components of $h$ in descending order of degree until a nonzero value is
encountered. In the worst case this touches all $M_h$ components of $h$ and thus costs
$O(M_h)$ arithmetic operations. 

Consequently, the cost of evaluating the certificate conditions along one
direction $d$ is
\[
O\!\left(M_f + \sum_{i=1}^m M_{g_i}\right),
\]
and running Algorithm~\ref{alg:main} on a sampled set $\mathcal{D}$ of size $|\mathcal{D}| = N$
requires
\[
O\!\left(
N\left( M_f + \sum_{i=1}^m M_{g_i} \right)
\right)
\]
operations.
Thus, the total runtime of the certificate scales linearly with the
number of sampled directions. 

\paragraph{Sample Complexity}
The certificate relies on generating a finite set
of directions $\mathcal{D} \subset \mathbb{S}^{n-1}$. Since correctness of
Algorithm~\ref{alg:main} does not depend on how $\mathcal{D}$ is chosen,
sampling affects only the likelihood of detecting a certifying direction when
one exists. The most natural strategy is to simply select each
$d \in \mathcal{D}$ independently from the uniform distribution on
$\mathbb{S}^{n-1}$; this choice of distribution is assumed to hold for the remainder of this paper. 
As sampling $N$ directions costs
$O(N)$ operations, whereas evaluating the asymptotic function at each direction requires
$O\!\left(N\cdot(M_f + \sum_{i=1}^m M_{g_i})\right)$ operations, the
overall runtime is dominated by polynomial evaluation.

\subsection{Statistical Guarantees under Uniform Random Sampling}\label{subsec:statistical}

In this section we upper bound the probability of the \textsc{Inconclusive} case. Define the set of \emph{certifying directions}
\[
\mathcal{D}^\star = \{ d \in \mathbb{S}^{n-1} : (g_i)_\infty < 0 \;\ \forall i \in [m]^+\}
\]
where $[m]^+ = \{0,1,\ldots, m\}$, the objective $f := g_0$ and the constraints $g_i$ for all $i \in [m]$ are as given by \eqref{POP}.  
Let
\[
    \alpha := \sigma(\mathcal{D}^\star)
\]
where $\sigma$ is the normalized Hausdorff measure
on $\mathbb{S}^{n-1}$ and is, by construction, a probability measure. As shown in Theorem~\ref{thm:directional-certificate},
non-emptiness of $\mathcal{D}^\star$ certifies that the problem is unbounded below.

Suppose now that Algorithm~\ref{alg:main} is run with
$\mathcal{D} = \{d_1,\dots,d_N\}$, where each $d_j \in \mathcal{D}$ is sampled i.i.d. uniformly from $\mathbb{S}^{n-1}$. For each $j$, define the indicator variable
\[
    X_j := \mathbbm{1}_{\{d_j \in \mathcal{D}^\star\}}.
\]
It follows that $\mathbb{P}(X_j = 1) = \alpha$ by definition of $\sigma$. The algorithm returns
\textsc{unbounded} if at least one $X_j = 1$, and returns \textsc{inconclusive}
when $X_1 = \cdots = X_N = 0$. Given $\alpha$, the probability of an inconclusive
outcome is therefore
\[
    \mathbb{P}(\textsc{inconclusive})
    = (1 - \alpha)^N.
\]
For any fixed $\alpha >0$ and confidence level $\delta \in (0,1)$ we can take 
\[
 N \ge \frac{\log \delta}{\log (1-\alpha)}
\]
to ensure that 
\[
(1-\alpha)^N \leq \delta.
\]
Thus, whenever the set of certifying directions occupies a positive-measure subset of $\mathbb{S}^{n-1}$, unboundedness is detected almost surely as $N \to \infty$. 

\section{Measure Zero Sets of Certifying Directions}

In Subsection \ref{subsec:statistical} we introduced the set of certifying directions 

\[
\mathcal{D}^\star = \{ d \in \mathbb{S}^{n-1} : (g_i)_\infty < 0 \;\ \forall i \in [m]^+\}
\]

letting $\alpha := \sigma(\mathcal{D}^\star)$ where $\sigma$ is again the normalized Hausdorff measure on $\mathbb{S}^{n-1}$. We refer to $\sigma$ as simply ``measure" henceforth. 
If $\mathcal{D}^\star$ is nonempty and $\alpha =0$ our certificate will be unable to sample a direction on $\mathbb{S}^{n-1}$ by which to assert unboundedness of the problem even though the problem is indeed unbounded. In this section we discuss how and why the measure zero case can occur with respect to \eqref{POP}. We begin by formalizing the notion of neighborhoods on the sphere. 

\begin{definition} (cf. pg 87 \cite{Munkres2000})
A set $U\subset\mathbb{S}^{n-1}$ is said to be relatively open (with respect to $\mathbb{S}^{n-1}$) if there exists an
open set $\widetilde U\subset\mathbb{R}^n$ such that
\[
U=\widetilde U\cap\mathbb{S}^{n-1}.
\]
\end{definition}
The above definition allows us to generalize the notion of an ``arc" or ``circular cap" on $\mathbb{S}^2$ or $\mathbb{S}^3$ respectively. Recall that each polynomial $g_i$ of degree $p_i$ admits the decomposition 
\begin{equation}\label{constraint_decomp}
g_i(x) = \sum_{k=1}^{p_i} \phi_{i,k}(x),
\end{equation}
where each $\phi_{i,k}$ is a homogeneous polynomial of degree $k$ and therefore the max degree term of $g_i$ is denoted $\phi_{i,p_i}$ for all $i \in [m]^+$. 

\begin{theorem}\label{thm: pos_measure}
Let $\mathcal{D}^\star$ be nonempty. If there exists $d_0 \in \mathcal{D}^\star$ such that $\phi_{i,p_i}(d_0) < 0$ for all $i \in [m]^+$, where $p_i:=\deg(g_i)$, then $\sigma(\mathcal{D}^\star)>0$.
\end{theorem}

\begin{proof}
    Note that each $\phi_{i,k}$ in \eqref{constraint_decomp} is continuous on
    $\mathbb{R}^n$, and therefore its restriction to $\mathbb{S}^{n-1}$ is continuous.
    Since $d_0 \in \mathcal{D}^\star$, we have
    \[
        (g_i)_\infty(d_0) < 0 \qquad \text{for all } i \in [m]^+.
    \]
    Together with the assumption $\phi_{i,p_i}(d_0) < 0$, Lemma~3.1 implies that
    $\mu_i(d_0) = p_i$. By continuity of $\phi_{i,p_i}$ on $\mathbb{R}^n$, for each
    $i \in [m]^+$ there exists an open set $\widetilde U_i \subseteq \mathbb{R}^n$
    with $d_0 \in \widetilde U_i$ such that
    \[
        \phi_{i,p_i}(d) < 0 \qquad \text{for all } d \in \widetilde U_i.
    \]
    Let $\widetilde U := \bigcap_{i=0}^m \widetilde U_i$ so that $\widetilde U$ is open in $\ren$. Further define
    \[
        U_i := \widetilde U_i \cap \mathbb{S}^{n-1}
    \]
    and $U := \bigcap_{i=0}^m U_i$ giving that
    \[
        U = \widetilde U \cap \mathbb{S}^{n-1}
    \]
    i.e., $U$ is relatively open. 
    For every $d \in U$ we have $\phi_{i,p_i}(d) < 0$ for all $i \in [m]^+$. In particular,
    $\mu_i(d)=p_i$ for all $i$ and all $d \in U$. Applying Lemma~3.1 once more, it follows that
    \[
        (g_i)_\infty(d) = -\infty < 0 \qquad \text{for all } d \in U \text{ and all } i \in [m]^+,
    \]
    and hence $U \subseteq \mathcal{D}^\star$. Finally, by the relative openness of $U$ there exists a nontrivial neighborhood of $d_0$ in $\mathbb{S}^{n-1}$, namely $\widetilde U$, implying that $U$ has positive measure. Therefore $\sigma(\mathcal{D}^\star) > 0$.
\end{proof}

The above result says that for the problem \eqref{POP} the certifying strict inequality $\phi_{i,k}(d) < 0$ must hold at the highest degree term $\phi_{i,p_i}$ for each constraint polynomial on a particular direction. Existence of such a direction allows one to claim that there exist directions in its immediate vicinity that also satisfy the certifying strict inequality thereby showing the existence of a relatively open subset of $\mathcal{D}^\star$. This gives us a natural condition by which to characterize the $\sigma(\mathcal{D}^\star)=0$ case. We first state some auxiliary lemmas towards the characterization result.

\begin{lemma}\label{lem:analytic_pullback}
Let $\phi:\mathbb{R}^n\to\mathbb{R}$ be a nonzero homogeneous polynomial and
$d^\ast\in\mathbb{S}^{n-1}$. Then there exist an open set
$V\subset\mathbb{R}^{n-1}$, a relatively open neighborhood
$U\subset\mathbb{S}^{n-1}$ of $d^\ast$, and a real-analytic map
$\Psi:V\to U$ such that $\phi\circ\Psi$ is real analytic on $V$,
not identically zero, and
\[
E:=\{u\in V:\phi(\Psi(u))=0\}
\]
has Lebesgue measure zero in $\mathbb{R}^{n-1}$.
\end{lemma}

\begin{proof}
The unit sphere can be written as
\[
\mathbb{S}^{n-1}=\{x\in\mathbb{R}^n : F(x)=0\},
\qquad
F(x):=\|x\|^2-1,
\]
where $F$ is real analytic and $\nabla F(x)=2x\neq 0$ for all $x\in\mathbb{S}^{n-1}$.
By the real-analytic implicit function theorem \cite[Thm.~1.8.3]{KrantzParks1992},
there exist an open set $V\subset\mathbb{R}^{n-1}$ and a real-analytic function
$\varphi: \ren \to\mathbb{R}$ such that
\[
U:=\{(u,\varphi(u)) : u\in V\}\subset\mathbb{S}^{n-1}
\]
is a relatively open neighborhood of $d^\ast$.
Define $\Psi:V\to U$ by $\Psi(u):=(u,\varphi(u))$.
If $\phi\circ\Psi\equiv 0$ on $V$, then $\phi$ vanishes on $U$ and, by homogeneity, on
\[
\{td : d\in U,\ t\in(1-\varepsilon,1+\varepsilon)\}
\]
for some $\varepsilon>0$. This set is a nonempty open subset of $\mathbb{R}^n$,
which forces $\phi\equiv 0$, a contradiction. Hence $\phi\circ\Psi\not\equiv 0$.
Since $\phi\circ\Psi$ is a nontrivial real-analytic function on $V$, its zero set
\[
E=\{u\in V : \phi(\Psi(u))=0\}
\]
has Lebesgue measure zero by Mityagin’s theorem \cite{Mityagin2015}.
\end{proof}

\begin{lemma}\label{lem:null_transport}
Let $V\subset\mathbb{R}^{n-1}$ be open, let $U\subset\mathbb{S}^{n-1}$ be
relatively open, and let $\Psi:V\to U$ be $C^1$.
If $E\subset V$ has Lebesgue measure zero, then
\[
\sigma(\Psi(E))=0.
\]
\end{lemma}

\begin{proof}
Let $\mathcal{H}^{n-1}$ denote the $(n-1)$-dimensional Hausdorff measure in $\mathbb{R}^n$.
Then the normalized surface measure $\sigma$ on $\mathbb{S}^{n-1}$ is given by
\[
\sigma(A)
:=
\frac{\mathcal{H}^{n-1}(A)}{\mathcal{H}^{n-1}(\mathbb{S}^{n-1})},
\qquad
A\subset\mathbb{S}^{n-1}\ \text{measurable}.
\]
Since $\Psi$ is $C^1$, it is locally Lipschitz. By the area formula
\cite[Thm.~3.8]{EvansGariepy}, the image under $\Psi$ of a Lebesgue measure zero set 
$E\subset\mathbb{R}^{n-1}$ has $\mathcal{H}^{n-1}$-measure zero in $\mathbb{R}^n$.
Therefore $\mathcal{H}^{n-1}(\Psi(E))=0$, and by the definition of $\sigma$,
\[
\sigma(\Psi(E))=0.
\]

\end{proof}

\begin{theorem}\label{thm:measure_zero}
    Suppose that for every direction $d \in \mathcal{D}^\star$ there exists an index $i \in [m]^+$ such that $\phi_{i,p_i}(d) = 0$ where $p_i = \text{deg}(g_i)$. Then 
    \[ \sigma(\mathcal{D}^\star) = 0. \]
\end{theorem}

\begin{proof}
By assumption,
\[
\mathcal{D}^\star \subseteq \bigcup_{i=0}^m Z_i,
\qquad
Z_i := \{ d \in \mathbb{S}^{n-1} : \phi_{i,p_i}(d)=0\}.
\]
It suffices to show that $\sigma(Z_i)=0$ for each $i\in[m]^+$. Fix $i\in[m]^+$. The polynomial $\phi_{i,p_i}$ is homogeneous, real analytic on $\mathbb{R}^n$,
and not identically zero. Let $d^\ast\in\mathbb{S}^{n-1}$ be arbitrary.
By Lemma~\ref{lem:analytic_pullback}, there exist an open set
$V\subset\mathbb{R}^{n-1}$, a relatively open neighborhood
$U\subset\mathbb{S}^{n-1}$ of $d^\ast$, and a real-analytic map
$\Psi:V\to U$ such that
\[
E:=\{u\in V:\phi_{i,p_i}(\Psi(u))=0\}
\]
has Lebesgue measure zero in $\mathbb{R}^{n-1}$.
By Lemma~\ref{lem:null_transport}, the surjectivity of $\Psi$ implies
\[
\sigma(Z_i\cap U)=\sigma(\Psi(E))=0.
\]
Since $d^\ast$ was arbitrary, we can define a collection of sets $\{U(d)\}_{d \in \mathbb{S}^{n-1}}$ which form an open cover of $\mathbb{S}^{n-1}$, where for each $d$,
$U(d)$ is the neighborhood obtained above. By compactness of $\mathbb{S}^{n-1}$, there exist
$d_1,\dots,d_N \in \mathbb{S}^{n-1}$ such that
\[
\mathbb{S}^{n-1} = \bigcup_{k=1}^N U(d_k).
\]
Set $U_k := U(d_k)$.
Hence
\[
Z_i=\bigcup_{k=1}^N (Z_i\cap U_k),
\]
and
\[
\sigma(\mathcal{D}^\star)
\le \sigma\Bigl(\bigcup_{i=0}^m Z_i\Bigr)
\le \sum_{i=0}^m \sigma(Z_i)\le \sum_{i=0}^m\sum_{k=1}^N \sigma(Z_i\cap U_k)=0,
\]
which proves the claim.
\end{proof}

To summarize the above: 
certification need only depend on existence of a direction $d \in \mathcal{D}^\star$ however for detection via sampling there must exist at least one direction for which there exists an active, or nonzero, component of maximum degree for every polynomial, including the objective. 
Below is a concrete example of a degenerate case for certification. 

\begin{example}
    \begin{equation}
\label{POP:case-study}
\begin{aligned}
\min_{x \in \mathbb{R}^2} \quad
& f(x) := (x_1^2 - x_2^2)^2 - x_2^3 \\
\text{s.t.} \quad
& g_1(x) := (x_1^2 - x_2^2)^2 - x_1^2 x_2^2 \le 0, \\
& g_2(x) := 1 - x_1^2 - x_2^2 \le 0 .
\end{aligned}
\end{equation}

The feasible set is noncompact and is illustrated in Figure~\ref{fig:feasible-set}. 
The constraint $g_2(x)\le 0$ enforces $\|x\|\ge 1$, while $g_1(x)\le 0$ restricts feasibility to conic regions centered about the lines $x_1=\pm x_2$. 
As a result, the feasible region consists of four unbounded conic sectors. 

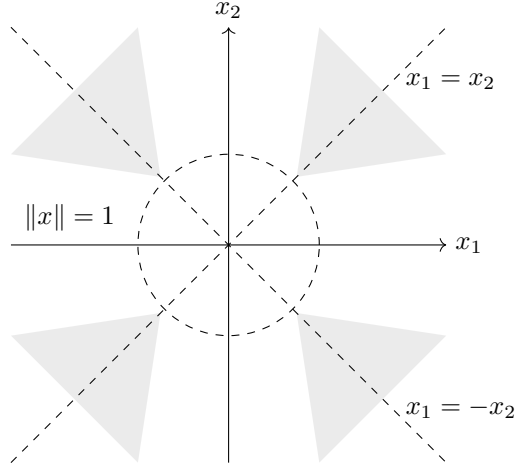
\begin{figure}[t]
\centering
\begin{tikzpicture}[scale=1.2]

\draw[->] (-2.4,0) -- (2.4,0) node[right] {$x_1$};
\draw[->] (0,-2.4) -- (0,2.4) node[above] {$x_2$};

\draw[dashed] (0,0) circle (1);

\fill[gray!25,opacity=0.6]
  (0.75,0.75) -- (2.4,1.0) -- (1.0,2.4) -- cycle;
\fill[gray!25,opacity=0.6]
  (-0.75,0.75) -- (-2.4,1.0) -- (-1.0,2.4) -- cycle;
\fill[gray!25,opacity=0.6]
  (0.75,-0.75) -- (2.4,-1.0) -- (1.0,-2.4) -- cycle;
\fill[gray!25,opacity=0.6]
  (-0.75,-0.75) -- (-2.4,-1.0) -- (-1.0,-2.4) -- cycle;

\draw[dashed] (-2.4,-2.4) -- (2.4,2.4);
\draw[dashed] (-2.4,2.4) -- (2.4,-2.4);

\node[anchor=west] at (1.85,1.8) {$x_1 = x_2$};
\node[anchor=west] at (1.85,-1.8) {$x_1 = -x_2$};
\node[anchor=east] at (-1.15,0.3) {$\|x\| = 1$};

\end{tikzpicture}
\caption{Feasible set of problem \eqref{POP:case-study}. 
The constraint $g_2(x)\le 0$ enforces $\|x\|\ge 1$, while 
$g_1(x)\le 0$ restricts feasibility to conic sectors centered about the reference directions $x_1=\pm x_2$ (dashed lines). 
Among these sectors, those with $x_2>0$ correspond to asymptotic descent directions satisfying $f_\infty(d)=-\infty$.}
\label{fig:feasible-set}
\end{figure}
We now certify that \eqref{POP:case-study} is unbounded from below using Theorem~\ref{thm:directional-certificate}. 
First observe that $g_1$ is homogeneous of degree four. 
Evaluating its leading term along the direction $d=(1,1)$ gives
\[
g_1(d) = (1-1)^2 - 1 = -1 < 0,
\]
and hence, by Lemma~\ref{lem:poly-asymptotics},
\[
(g_1)_\infty(d) = -\infty.
\]
Similarly, $g_2(x)=1-x_1^2-x_2^2$ has leading homogeneous term $-(x_1^2+x_2^2)$, so
\[
(g_2)_\infty(d) = -\infty
\qquad \text{for all } d\neq 0.
\]
Thus $d=(1,1)$ is an asymptotically feasible direction satisfying the conditions of Theorem~\ref{thm:directional-certificate}.
Next, decompose the objective into homogeneous components,
\[
f(x)=\phi_4(x)+\phi_3(x),
\qquad
\phi_4(x)=(x_1^2-x_2^2)^2,\quad \phi_3(x)=-x_2^3.
\]
For $d=(1,1)$ we have $\phi_4(d)=0$ and $\phi_3(d)=-1<0$, and therefore, by Lemma~\ref{lem:poly-asymptotics},
\[
f_\infty(d)=-\infty.
\]
Since $f_\infty(d)<0$ and $(g_i)_\infty(d)<0$ for $i=1,2$, the hypotheses of Theorem~\ref{thm:directional-certificate} are satisfied and \eqref{POP:case-study} is unbounded however $\sigma(\mathcal{D}^\star) = 0.$ 
\end{example}

\section{Implications for Global Polynomial Optimization Solvers}

Modern global solvers for polynomial optimization reason about unboundedness \textit{during} global search, not prior. Sum-of-Squares and moment-based hierarchies check unboundedness by solving a sequence of increasingly large but finite relaxations, inferring unboundedness only indirectly when lower bounds fail to stabilize. Similarly, branch-and-bound and branch-and-reduce algorithms reason over bounded regions of the feasible set, either by construction or through dynamically generated bounding boxes, and may terminate only after extensive global exploration or numerical heuristics suggest divergence. The limitation of these methods is summarized in the following remark. 

\begin{remark}[Finite-region indistinguishability]
\label{rem:finite_truncation}
There exist polynomial optimization problems for which
\[
\inf_{x \in X} f(x) = -\infty,
\]
but such that for every radius $R > 0$ the truncated problem
\[
\inf \{ f(x) : x \in X,\ \|x\| \le R \}
\]
admits a finite optimal value. In this situation, unboundedness is never witnessed on any bounded subset of the feasible region. 
\end{remark}

The certification's approach of explicitly testing asymptotic feasibility and asymptotic descent along directions on the unit sphere reasons directly about behavior at infinity, rather than inferring it indirectly from bounded approximations. When certifying directions occupy a set of positive measure the certificate exposes a concrete geometric mechanism—an explicit direction of feasible descent—that is invisible to purely region-based reasoning.

From a solver perspective, this suggests a natural division of labor. Directional certification may be applied as a preprocessing step, prior to global optimization, to either (i) certify unboundedness outright by identifying a feasible descent ray, or (ii) return an inconclusive outcome that carries meaningful geometric information, namely that any unboundedness mechanism must be directionally fragile and therefore undetectable by finite truncation alone. In the latter case, solvers may proceed with bounded relaxations or branching strategies with a clearer understanding of the structural regime they are operating in. Importantly, the certificate does not replace existing solvers, nor does it attempt to integrate into their internal logic; rather, it provides asymptotic information that complements finite-region methods and clarifies the geometric source of their successes and limitations.

Viewed in this light, the contribution of the present work is not merely a new test for unboundedness, but a geometric framework for interpreting when and why unboundedness is algorithmically detectable. By separating directionally robust mechanisms from undectable ones, the certificate explains observed solver behavior, delineates the limits of finite-region reasoning, and provides a principled tool for exposing unboundedness before expensive global search is undertaken.

\subsubsection*{Data Availability}
Data sharing not applicable to this article as no datasets were generated or analysed during the current study

\subsection*{Funding and Conflict of Interest Statement} 
The authors received partial support from the Arizona State University, no external funding organizations supported this study. The authors declare they have no financial interests in the outcome of this research. 

\bibliography{sn-bibliography}
\end{document}